\documentclass[11pt]{amsart}

\usepackage{amsmath,amssymb,amsthm,mathtools}
\usepackage[hyphens]{url}
\usepackage[usenames,dvipsnames]{color}
\usepackage[pagebackref,linktocpage=true,colorlinks=true,linkcolor=RoyalBlue,citecolor=BrickRed,urlcolor=RoyalBlue]{hyperref}
\usepackage[msc-links,abbrev]{amsrefs}
\usepackage{stmaryrd}

\newcommand{\R}{\mathbf R}
\newcommand{\Sph}{\mathbf S}
\newcommand{\G}{\mathcal G}
\newcommand{\Ric}{\mathrm{Ric}}
\newcommand{\Scal}{\mathrm{Scal}}
\newcommand{\Pf}{\operatorname{Pf}}
\newcommand{\tr}{\operatorname{tr}}
\newcommand{\diver}{\operatorname{div}}
\newcommand{\diam}{\operatorname{diam}}
\newcommand{\TPf}{\operatorname{TPf}}

\newtheorem{theorem}{Theorem}[section]
\newtheorem{proposition}[theorem]{Proposition}
\newtheorem{lemma}[theorem]{Lemma}
\newtheorem{corollary}[theorem]{Corollary}

\theoremstyle{definition}
\newtheorem{note}[theorem]{Note}

\begin{document}

\title[Total curvature and isoperimetric inequalities]
{Total curvature and isoperimetric inequalities\\in pinched Cartan--Hadamard manifolds}

\author{Mohammad Ghomi}
\address{School of Mathematics, Georgia Institute of Technology, Atlanta, GA 30332}
\email{ghomi@math.gatech.edu}
\urladdr{www.math.gatech.edu/\textasciitilde ghomi}

\begin{abstract}
We establish a sharp lower bound for the total Gauss--Kronecker
curvature of convex hypersurfaces in Cartan--Hadamard manifolds with
pinched negative curvature. The bound holds in all dimensions when
the diameter is small relative to the curvature scale, and in
dimensions $4$ and $5$ without any restriction on the diameter,
provided that the pinching is sufficiently tight. The proofs are
based on the Chern--Gauss--Bonnet theorem and weighted
Hsiung--Minkowski inequalities. As an application, we obtain the
isoperimetric inequality of the Cartan--Hadamard conjecture in
dimension $5$ under sufficiently pinched curvature.
\end{abstract}

\renewcommand{\subjclassname}{\textup{2020} Mathematics Subject Classification}
\date{Last revised on \today}

\maketitle

%%%%%%%%%%%%%%%%%%%%%%%%%%%%%%%%%%%%%%%%%%
\section{Introduction}\label{sec:intro}
%%%%%%%%%%%%%%%%%%%%%%%%%%%%%%%%%%%%%%%%%%

A Cartan--Hadamard manifold $M^n$ is a complete simply connected
Riemannian space with sectional curvature $K\leq 0$. A convex
hypersurface $\Gamma\subset M$ is the boundary of a compact convex set
with nonempty interior. If $\Gamma$ is smooth, its total Gauss--Kronecker curvature $\G(\Gamma)$ is  the integral of the determinant of its second fundamental form. In hyperbolic space $\mathbf{H}^n$,
\begin{equation}\label{eq:G}
\G(\Gamma)\geq |\Sph^{n-1}|,
\end{equation}
where $|\Sph^{n-1}|$ denotes the volume of the unit sphere
\cites{borbely2002,ghomi-spruck2022}. For $n\leq3$ the same inequality
holds in every Cartan--Hadamard manifold, as an immediate consequence of
the Gauss--Bonnet theorem and the Gauss equation. Whether \eqref{eq:G} holds in higher dimensions is a long-standing problem \cites{willmore-saleemi1966,bgs1985}, which
implies the Cartan--Hadamard conjecture \cites{gromov1999, aubin1975, burago-zalgaller1988} on the isoperimetric inequality
\cite{ghomi-spruck2022}. We establish \eqref{eq:G} when $\diam(\Gamma)\coloneqq\sup_{x,y\in\Gamma}\operatorname{dist}_M(x,y)$
 is sufficiently small relative to the bounds of $K$:

\begin{theorem}\label{thm:diameter}
 For $n\geq 4$, let $M^n$ be a Cartan--Hadamard manifold with curvature
$
-b^2\leq K\leq-a^2<0.
$
Then  there exists $\delta_n>0$, depending only on $n$, such that \eqref{eq:G} holds
for every smooth convex hypersurface $\Gamma\subset M$ with 
$\diam(\Gamma)\leq\delta_na/b^2$.
\end{theorem}

The proof yields a stronger quantitative estimate; see
Proposition~\ref{prop:quantitative-diameter}.
In dimensions $4$ and $5$ the diameter restriction can be removed as follows.
Set
$$
\lambda_4\coloneqq\frac{1}{\sqrt3}\approx0.58,
\qquad\quad
\lambda_5\coloneqq\frac{1+\sqrt{113}}{14}\approx0.83.
$$

\begin{theorem}\label{thm:pinched}
For $n=4$, $5$, let $M^n$ be a Cartan--Hadamard manifold with curvature
$
-b^2\leq K\leq-\lambda_nb^2,
$
for some $b>0$.
Then \eqref{eq:G} holds for every smooth convex hypersurface $\Gamma\subset M^n$.
\end{theorem}

Again the proof yields explicit lower bounds for the defect in \eqref{eq:G}; see Propositions~\ref{prop:dim4} and~\ref{prop:dim5}.
These theorems provide the first instances of
\eqref{eq:G} in dimensions greater than $3$ for Cartan--Hadamard manifolds of nonconstant negative curvature near $\Gamma$.
In light of the reduction of the isoperimetric inequality to the total
curvature inequality \cite[Thm.~7.1]{ghomi-spruck2022}, Theorem~\ref{thm:pinched} immediately yields the following result,
which gives the first instance of the Cartan--Hadamard conjecture in a
dimension greater than $4$ with nonconstant
negative curvature. For any set $S\subset M^n$, let $|S|$ and $|\partial S|$ denote its volume and perimeter respectively.

\begin{corollary}\label{cor:isoperimetric}
Let $M^5$ be a Cartan--Hadamard manifold with curvature
$
-b^2\leq K\leq-\lambda_5b^2,
$
for some $b>0$.
Then every bounded set $\Omega\subset M$ of finite perimeter and positive volume satisfies $|\partial\Omega|>|\partial B|$, where $B$
is a ball in $\R^5$ with  $|B|=|\Omega|$. 
\end{corollary}

For general Cartan--Hadamard manifolds, the isoperimetric inequality has
been established only in dimensions $n\leq4$
\cites{weil1926,kleiner1992,croke1984}, while the total curvature
inequality \eqref{eq:G} is known only for $n\leq3$. For domains of small volume, related isoperimetric inequalities
are known under additional compactness, bounded-geometry, or
asymptotic hypotheses
\cites{morgan-johnson2000,druet2002,nardulli-osorio2020}. The pinching assumptions above do not force $K$ to be
constant: for all $0<a<b$ there exist Cartan--Hadamard manifolds with
$-b^2\leq K\leq-a^2$ which are not of constant curvature, such as the
universal covers of the closed manifolds constructed by Gromov--Thurston \cite{gromov-thurston1987}. The constant in \eqref{eq:G} is
sharp, since the total curvature of geodesic
spheres converges to $|\Sph^{n-1}|$ as their radii shrink to zero \cite[Lem.~5.1]{ghomi-spruck2022}. 
The total curvature inequality \eqref{eq:G} has been established recently in the case where $M^n$ has nullity at least $n-3$ \cite{ghomi2026-nullity} or the curvature  is constant near $\Gamma$ \cite{ghomi-stavroulakis2026b}. See 
 \cites{ghomi-stavroulakis2026,schulze2020,ghomi-spruck2023c} for some other results in this area, and \cites{ghomi-spruck2022, kloeckner-kuperberg2019} for more references and background.

The proofs combine two main ingredients. First, we obtain an expression for the defect $\G(\Gamma)-|\Sph^{n-1}|$ by expanding the Chern--Gauss--Bonnet integrands according to the number of ambient curvature
factors (Section~\ref{sec:cgb}). The term linear
in the ambient curvature has a favorable sign, while the higher-order
terms must be controlled by total mean curvatures. Second, we develop a weighted
Hsiung--Minkowski identity in terms of Newton transformations and
apply it, together with Hessian comparison and Stokes theorem, to suitable radial
functions (Section~\ref{sec:newton}). The resulting inequalities provide the required controls. 
Choosing the squared distance from an interior point of the convex body bounded by $\Gamma$ yields the estimates used
to absorb the higher-order terms when the diameter is small, and leads
to Theorem~\ref{thm:diameter} (Section~\ref{sec:diameter}). Choosing
the distance from an exterior point gives diameter-free estimates, which lead to
Theorem~\ref{thm:pinched} (Section~\ref{sec:low-dim}).

%%%%%%%%%%%%%%%%%%%%%%%%%%%%%%%%%%%%%%%%%%
\section{Algebraic Preliminaries}\label{sec:prelim}
%%%%%%%%%%%%%%%%%%%%%%%%%%%%%%%%%%%%%%%%%%

We begin by collecting the pointwise estimates needed for curvature
tensors in the proofs below. An \emph{algebraic curvature tensor} on
$\R^n$ is a $4$-tensor $R$ satisfying the symmetries
$$
R(x,y,z,w)=R(z,w,x,y),
\qquad
R(x,y,z,w)=-R(y,x,z,w)=-R(x,y,w,z),
$$
$$
R(x,y,z,w)+R(y,z,x,w)+R(z,x,y,w)=0.
$$
 For orthonormal vectors $x,y$, the \emph{sectional curvature} of
$R$ is defined by
$$
K(x,y)\coloneqq R(x,y,x,y).
$$
Furthermore, the  \emph{Ricci} and
\emph{scalar curvatures} are given by
$$
\Ric_R(x,y)\coloneqq\sum_{i=1}^nR(e_i,x,e_i,y),
\qquad
\Scal_R\coloneqq\sum_{i=1}^n\Ric_R(e_i,e_i),
$$
 where $e_1,\dots,e_n$ is the standard basis of $\R^n$. We write
$
R_{ijk\ell}\coloneqq R(e_i,e_j,e_k,e_\ell),
$
and
$
K_{ij}\coloneqq K(e_i,e_j)=R_{ijij}.
$
 The inner product and norm on tensors are those induced by the
Euclidean inner product on $\R^n$. Thus, for $k$-tensors $S$ and $T$,
we have
$$
\langle S,T\rangle
\coloneqq
\sum_{i_1,\dots,i_k}
S(e_{i_1},\dots,e_{i_k})
T(e_{i_1},\dots,e_{i_k}),
\qquad
|T|\coloneqq\langle T,T\rangle^\frac{1}{2}.
$$
 Throughout the paper, $c_n$ denotes a positive constant depending only
on $n$, which may change from line to line.

\subsection{General estimates}
First we record the curvature estimates which hold in all dimensions.
 Let $\mathcal R_n$ denote the space of algebraic curvature tensors on
$\R^n$, and let $R_k$ be the tensor of constant sectional curvature
$k$.

\begin{lemma}\label{lem:algebraic}
Let $R\in\mathcal R_n$ with
$
-1\leq K\leq-\lambda.
$
Then
$
|R-R_{-1}|\leq c_n(1-\lambda).
$
\end{lemma}

\begin{proof}
Let
$
\|R\|_K\coloneqq\sup_{x,y}|K(x,y)|,
$
where the supremum is over all orthonormal pairs $x,y\in\R^n$.
An algebraic curvature tensor is determined by its sectional
curvatures through polarization
\cite[p.~14]{cheeger-ebin2008}. Thus
$\|R\|_K=0$ implies $R=0$. Since $\|\cdot\|_K$ is also homogeneous
and satisfies the triangle inequality, it is a norm on $\mathcal R_n$.
As $\mathcal R_n$ is finite-dimensional, this norm is equivalent to
the standard tensor norm. Since the sectional curvature of $R-R_{-1}$ lies between
$0$ and $1-\lambda$, the result follows.
\end{proof}

We also need the following estimate of Berger \cites{berger1960,karcher1970}.

\begin{lemma}[Berger]\label{lem:berger}
Let $R\in\mathcal R_n$ with $0\leq K\leq\Lambda$. Then
$
|R_{ijjk}|\leq\Lambda/2
$
for distinct $i$, $j$, $k$.
\end{lemma}

For even $n$, define the \emph{Pfaffian} of $R\in\mathcal R_n$ by
\begin{equation}\label{eq:PfR}
\Pf(R)
\coloneqq
\frac1{2^{n/2}n!}
\sum
\epsilon_{i_1\cdots i_n}\epsilon_{j_1\cdots j_n}
R_{i_1i_2j_1j_2}\cdots
R_{i_{n-1}i_nj_{n-1}j_n},
\end{equation}
where the sum is over all indices from $1$ to $n$, and
$\epsilon_{i_1\cdots i_n}$ is the Levi-Civita symbol, equal to the sign
of the permutation $(i_1,\dots,i_n)$ when the indices are distinct, and
to zero otherwise. The normalizing constant ensures that $\Pf(R_1)=1$.

For a symmetric linear map $A\colon\R^n\to\R^n$, let $\sigma_j(A)\coloneqq\sigma_j(\kappa_1,\dots, \kappa_n)$ be
the $j^{\mathrm{th}}$ elementary symmetric function of the eigenvalues $\kappa_i$
of $A$, with $\sigma_0(A):=1$, and write
$
A_{ij}\coloneqq\langle Ae_i,e_j\rangle.
$
For $R\in\mathcal R_n$, define $R^A$ via the Gauss equation
\begin{equation}\label{eq:RA}
(R^A)_{ijk\ell}
\coloneqq
R_{ijk\ell}+A_{ik}A_{j\ell}-A_{i\ell}A_{jk}.
\end{equation}
Equivalently, $R^A\coloneqq R+\tfrac12 A\owedge A$, where 
$\owedge$
 is the Kulkarni--Nomizu product. Let
$$
GK\coloneqq\det A
$$
be the \emph{Gauss--Kronecker curvature} of $A$. When $A$ is diagonal, we set
$$
L\coloneqq\sum_{i<j}K_{ij}\prod_{\ell\neq i,j}\kappa_\ell.
$$
The next computation gives an algebraic decomposition of 
$\Pf(R^A)$, together with an estimate for its remainder, for use when the
Gauss equation is substituted into the Chern--Gauss--Bonnet formulas below.

\begin{lemma}\label{lem:linear-term}
Let $n$ be even, $A\colon\R^n\to\R^n$ be a diagonal
linear map with entries $\kappa_1,\dots,\kappa_n$, and
$R\in\mathcal R_n$. Then
$$
\Pf(R^A)
=
GK
+
\frac1{n-1}L
+
E_n.
$$
If $A$ is positive semidefinite, then
\begin{equation}\label{eq:E}
|E_n|
\leq
c_n
\sum_{s=2}^{n/2}
|R|^s\sigma_{n-2s}(A).
\end{equation}
Furthermore, $E_4=\Pf(R)$.
\end{lemma}

\begin{proof}
Let $Q\coloneqq\tfrac12 A\owedge A$, that is,
$
Q_{ijk\ell}
=
A_{ik}A_{j\ell}-A_{i\ell}A_{jk}$. So $R^A=R+Q$. 
By \eqref{eq:PfR},
$
\Pf(R+Q)
=
\sum_{s=0}^{n/2}P_s,
$
where $P_s$ is the sum of all terms containing exactly
$s$ factors of $R$ and $n/2-s$ factors of $Q$.
Since $A$ is diagonal, the nonzero components of $Q$ are, up to
symmetries,
$
Q_{ijij}=\kappa_i\kappa_j.
$
Hence in \eqref{eq:PfR} each nonzero term of $P_0$ pairs the
indices and contributes $\kappa_1\cdots\kappa_n$, so
$
P_0=\Pf(Q)=\det A=GK.
$

Each term in $P_1$ contains one factor
of $R$ and $n/2-1$ factors of $Q$. Since the latter are
nonzero only when their indices occur in equal pairs, the factor of
$R$ is of the form $R_{ijij}=K_{ij}$. Thus
$
P_1=\mu_nL.
$
To determine $\mu_n$, take $A=tI$ and
$R=\tau R_1$. Then $R^A=(\tau+t^2)R_1$, and hence
$
\Pf(R^A)=(\tau+t^2)^{n/2}.
$
Comparing the terms linear in $\tau$ gives
$
\frac n2t^{n-2}
=
\mu_n\binom n2t^{n-2},
$
so $\mu_n=1/(n-1)$.

For $s\geq2$, each term in $P_s$ contains $s$
factors of $R$ and $n-2s$ eigenvalues of $A$ with distinct
indices. Setting $E_n\coloneqq\sum_{s\geq2}P_s$ gives the stated
decomposition. If $A$ is positive semidefinite, then
$\kappa_\ell\geq0$, and
$
|P_s|
\leq
c_n|R|^s\sigma_{n-2s}(A).
$
Summing these estimates gives \eqref{eq:E}. Finally, when $n=4$, the only term with $s\geq2$ is
$
P_2=\Pf(R),
$
so $E_4=\Pf(R)$.
\end{proof}

\subsection{Estimates in dimension 4}
When $n=4$, we have \cite[(6.31)]{besse1987}
\begin{equation}\label{eq:Pf-dim4}
\Pf(R)
=
\frac1{24}
\Big(
|R|^2-4|\Ric_R|^2+\Scal_R^2
\Big).
\end{equation}
To estimate this quantity, let $\Lambda^2\R^4$ be the space of
$2$-forms on $\R^4$, with the inner product
$
\langle x\wedge y,z\wedge w\rangle
\coloneqq
\langle x,z\rangle\langle y,w\rangle-\langle x,w\rangle\langle y,z\rangle,
$
and norm $|\omega|\coloneqq\langle\omega,\omega\rangle^{1/2}$. The
\emph{curvature operator} of $R\in\mathcal R_4$ is the self-adjoint
map $\mathcal R\colon\Lambda^2\R^4\to\Lambda^2\R^4$ given by
$$
\langle\mathcal R(x\wedge y),z\wedge w\rangle
\coloneqq
R(x,y,z,w).
$$
 Let
$\Lambda^\pm\subset\Lambda^2\R^4$ be the eigenspaces, corresponding to
the eigenvalues $\pm1$, of the \emph{Hodge star operator}, i.e., the
involution $*\colon\Lambda^2\R^4\to \Lambda^2\R^4$ determined by
$\alpha\wedge{*\beta}=\langle\alpha,\beta\rangle\,dV$, where
$dV\coloneqq e_1\wedge e_2\wedge e_3\wedge e_4$. Then $\Lambda^2\R^4=\Lambda^+\oplus\Lambda^-$, and
elements of $\Lambda^+$, $\Lambda^-$ are called \emph{self-dual} and
\emph{anti-self-dual} forms respectively. With respect to this
splitting, $\mathcal R$ assumes the block form
$$
\mathcal R
=
\begin{pmatrix}
\mathcal R^+ & Z\\
Z^{\mathrm T} & \mathcal R^-
\end{pmatrix},
$$
where $|Z|^2=\frac14|\Ric_R-(\Scal_R/4)\,\langle\cdot,\cdot\rangle|^2$
\cite{besse1987}. A unit form $\omega\in\Lambda^2\R^4$ is
\emph{decomposable}, i.e., $\omega=x\wedge y$ for some vectors $x$,
$y\in\R^4$ which may then be taken to be orthonormal, if and only if
$\omega\wedge\omega=0$. For such $\omega$,
$$
\langle\mathcal R\,\omega,\omega\rangle
=
R(x,y,x,y)
=
K(x,y).
$$
Since
$\omega\wedge\omega=\langle\omega,*\omega\rangle\,dV
=\big(|\omega^+|^2-|\omega^-|^2\big)dV$ for the components
$\omega^\pm$ of $\omega$ in $\Lambda^\pm$, a unit form $\omega$ is
decomposable if and only if $|\omega^+|=|\omega^-|=1/\sqrt2$. Using
these facts, we obtain the following estimate, which refines an earlier
result of Ville \cite{ville1987}; see Note~\ref{note:ville}.

\begin{lemma}\label{lem:pf-bound}
If $R\in\mathcal R_4$ with
$
-1\leq K\leq0,
$
then
$$
\Pf(R)\leq1-\Big(1+\frac{\Scal_R}{12}\Big)^2\leq1.
$$
Furthermore, $\Pf(R)=1$ if and only if $R=R_{-1}$.
\end{lemma}

\begin{proof}
Since $|R|^2=4|\mathcal R|^2=4(|\mathcal R^+|^2+|\mathcal R^-|^2+2|Z|^2)$, while
$4|\Ric_R|^2-\Scal_R^2=16|Z|^2$, \eqref{eq:Pf-dim4} takes the form
$$
\Pf(R)
=
\frac16\big(|\mathcal R^+|^2+|\mathcal R^-|^2\big)-\frac13|Z|^2
\leq
\frac16\big(|\mathcal R^+|^2+|\mathcal R^-|^2\big).
$$
Note that $\tr\mathcal R=\sum_{i<j}R_{ijij}=\Scal_R/2$. Also,
since $*=\pm1$ on $\Lambda^\pm$,
$$
\tr\mathcal R^+-\tr\mathcal R^-
=
\tr(\mathcal R\circ *)
=
2\big(R_{1234}+R_{1342}+R_{1423}\big)
=
0,
$$
where the middle equality is seen by computing the trace in the basis
$e_i\wedge e_j$, and the last one is the first Bianchi
identity. Hence $\tr\mathcal R^\pm=\tfrac12\tr\mathcal R=\Scal_R/4=3t$, where
$t\coloneqq\Scal_R/12$. Let $\alpha_i^\pm$ be the eigenvalues of
$\mathcal R^\pm$, with unit eigenvectors
$\omega_i^\pm\in\Lambda^\pm$. The unit forms
$(\omega_i^+\pm\omega_j^-)/\sqrt2$ are decomposable, and the
corresponding sectional curvatures are
$
(
\alpha_i^++\alpha_j^-
\pm2\langle Z\omega_j^-,\omega_i^+\rangle
)/2.
$
Their average is $(\alpha_i^++\alpha_j^-)/2$. Thus
$w_{ij}\coloneqq\alpha_i^++\alpha_j^-\in[-2,0]$, which yields
$$
\sum_{i,j=1}^3 w_{ij}\big(w_{ij}+2\big)\leq0.
$$
Since $\sum_{ij}w_{ij}=3\tr\mathcal R^++3\tr\mathcal R^-=18t$, and
$\sum_{ij}w_{ij}^2=3\sum_i\big((\alpha_i^+)^2+(\alpha_i^-)^2\big)+2\,\tr\mathcal R^+\tr\mathcal R^-$,
the last inequality may be rewritten as
$
\sum_i\big((\alpha_i^+)^2+(\alpha_i^-)^2\big)\leq-12t-6t^2.
$
Hence
$$
\Pf(R)
\leq
\frac16\sum_i\Big((\alpha_i^+)^2+(\alpha_i^-)^2\Big)
\leq
-2t-t^2
=
1-(1+t)^2,
$$
which establishes the desired estimate.
Finally, if $\Pf(R)=1$, then equality holds throughout: $Z=0$, $t=-1$,
and $w_{ij}\in\{0,-2\}$ for all $i$, $j$. Since
$\sum_{ij}w_{ij}=18t=-18$, it follows that $w_{ij}=-2$ for all $i$,
$j$, which forces $\alpha_i^\pm=-1$ for all $i$. So $\mathcal R=-I$, or
$R=R_{-1}$. Conversely, $|R_{-1}|^2=24$, $|\Ric_{R_{-1}}|^2=36$, and
$\Scal_{R_{-1}}=-12$, so $\Pf(R_{-1})=1$ by \eqref{eq:Pf-dim4}.
\end{proof}

\begin{note}\label{note:ville}
The  inequality $\Pf(R)\leq1$ in Lemma~\ref{lem:pf-bound}, together with its
equality case, is due to Ville \cite[Thm.~1]{ville1987}, who showed that
the constant curvature tensor maximizes the Gauss--Bonnet integrand,
which is the Pfaffian up to a normalizing factor. Integration via the
Chern--Gauss--Bonnet theorem then gives
$
\chi(M)\leq 3\textup{vol}(M)/(4\pi^2)
$
for closed Riemannian $4$-manifolds with $-1\leq K\leq0$, with equality
if and only if $K\equiv -1$; see also Bettiol--Kummer--Mendes
\cite[Thm.~E]{bettiol-kummer-mendes2024}, where this inequality and related results are obtained via pointwise curvature estimates.
\end{note}

%%%%%%%%%%%%%%%%%%%%%%%%%%%%%%%%%%%%%%%%%%
\section{Chern--Gauss--Bonnet Defect Formulas}\label{sec:cgb}
%%%%%%%%%%%%%%%%%%%%%%%%%%%%%%%%%%%%%%%%%%

Here we use the Chern--Gauss--Bonnet theorem to compute the defect of
the total curvature from the sharp bound in \eqref{eq:G}; see
Proposition~\ref{prop:defect} below. To this end, let $N$ be a
compact oriented Riemannian $n$-manifold with smooth boundary
$\partial N$, oriented by the outward unit normal $\nu$. Let $R$ and $K$ denote the curvature tensor and sectional
curvature of $N$, and let
$$
A(X)\coloneqq\nabla_X\nu
$$
be the shape operator of $\partial N$.
We recall the formulation of the Chern--Gauss--Bonnet theorem
\cite{chern1945} given in \cite[Sec.~2]{ghomi2026-nullity}. Let $e_1,\dots,e_{n-1}$ be a positively
oriented local orthonormal frame on $\partial N$, and define
$$
\alpha_i(X)\coloneqq\langle A(X),e_i\rangle,
\qquad
\Omega_{ij}(X,Y)\coloneqq R(X,Y,e_i,e_j).
$$
If $e_i$ is a principal direction of $\partial N$, then
$\alpha_i(e_i)=\kappa_i$, where $\kappa_i$ is the corresponding
principal curvature. Thus the $1$-forms $\alpha_i$ record the second
fundamental form of $\partial N$. The $2$-forms $\Omega_{ij}$, which we
call the \emph{ambient curvature forms}, record the tangential
components of the curvature of $N$. Suppose now that $n$ is even.
Then, for $0\leq s\leq n/2-1$, the Chern $(n-1)$-forms are given by combining these forms as follows
\begin{multline*}
\Phi_s\coloneqq
\sum_{\sigma\in S_{n-1}}\operatorname{sgn}(\sigma)
\Big(
\alpha_{\sigma(1)}\wedge\cdots\wedge
\alpha_{\sigma(n-1-2s)}
\Big)\\
\wedge
\Big(
\Omega_{\sigma(n-2s)\sigma(n-2s+1)}
\wedge\cdots\wedge
\Omega_{\sigma(n-2)\sigma(n-1)}
\Big),
\end{multline*}
where the second factor is omitted when $s=0$. The boundary term, which appears in the Chern--Gauss--Bonnet theorem for manifolds with boundary,
is given by the transgression formula
\begin{equation}
\TPf_{\partial N}\label{eq:transgression}
\coloneqq
\sum_{s=0}^{n/2-1}
c_{n,s}\Phi_s(e_1,\dots,e_{n-1}),
\end{equation}
where
$
c_{n,s}
\coloneqq
|\Sph^{n-1}|/
((4\pi)^s s!\,|\Sph^{n-1-2s}|(n-1-2s)!).
$
These normalizing constants ensure that when
the ambient curvature vanishes,
$
\TPf_{\partial N}=GK.
$
The Chern--Gauss--Bonnet theorem states that, when $n$ is even,
\begin{equation}\label{eq:CGB}
|\Sph^{n-1}|\chi(N)
=
\frac{2|\Sph^{n-1}|}{|\Sph^n|}
\int_N\Pf(R)
+
\int_{\partial N}\TPf_{\partial N}.
\end{equation}

At each point of $\partial N$, choose the frame $e_1,\dots,e_{n-1}$ to be
\emph{principal}, i.e., to consist of principal directions, with
$\kappa_1,\dots,\kappa_{n-1}$ the corresponding principal curvatures. Set $K_{ij}\coloneqq K(e_i,e_j)$, and recall that
$GK\coloneqq\det A$ as in Section~\ref{sec:prelim}. The following result isolates the terms with at most one ambient curvature form in $\TPf_{\partial N}$ and gives an estimate for the remaining terms. For the terms with one curvature form, or the linear portion of $\TPf$, recall that
$
L=\sum_{i<j}K_{ij}\prod_{\ell\neq i,j}\kappa_\ell,
$
as in Section~\ref{sec:prelim}. 

\begin{lemma}\label{lem:TPf}
Suppose that $n\geq 4$ is even. Then
$$
\TPf_{\partial N}
=
GK
+
\frac1{n-2}
L
+
E_{n-1},
$$
where $E_{n-1}$ consists of the terms containing at least two ambient curvature
forms. If $A$ is positive semidefinite, then
$
|E_{n-1}|
\leq
c_n\sum_{s=2}^{n/2-1}|R|^s\sigma_{n-1-2s}(A).
$
\end{lemma}

\begin{proof}
Let $\theta^1,\dots,\theta^{n-1}$ be the coframe dual to
$e_1,\dots,e_{n-1}$. Since the frame is principal,
$\alpha_i=\kappa_i\theta^i$. Consequently,
$$
c_{n,0}\Phi_0(e_1,\dots,e_{n-1})
=
\kappa_1\cdots\kappa_{n-1}
=
GK,
$$
and we obtain the first term in the stated formula.
The term containing one curvature form is 
$$
c_{n,1}\Phi_1(e_1,\dots,e_{n-1})
=
c_{n,1} 2(n-3)!L.
$$
Indeed, since $\alpha_i=\kappa_i\theta^i$, in the evaluation of each
summand of $\Phi_1$ on $(e_1,\dots,e_{n-1})$ the $1$-forms must receive
the vectors $e_{\sigma(1)},\dots,e_{\sigma(n-3)}$; so  $\Omega_{\sigma(n-2)\sigma(n-1)}$ is evaluated on the remaining
pair, which yields  $K_{\sigma(n-2)\sigma(n-1)}$.
Once the pair $\{i,j\}$ in $\Omega_{ij}$ is fixed, the
remaining indices may be ordered in $(n-3)!$ ways, while the two
orders of $i,j$ give the same contribution. Since
$$
c_{n,1}2(n-3)!
=
\frac{|\Sph^{n-1}|}{2\pi|\Sph^{n-3}|}
=
\frac1{n-2},
$$
we obtain the second term in $\TPf_{\partial N}$.
Every remaining term contains $s\geq2$ tangential components of $R$
and $n-1-2s$ principal curvatures with distinct indices. If
$\kappa_i\geq0$, the same counting argument used to obtain
\eqref{eq:E} gives the stated estimate for $E_{n-1}$.
\end{proof}

The preceding computation and the Chern--Gauss--Bonnet theorem yield
the following defect formulas. Set $\G(\partial N)\coloneqq\int_{\partial N} GK$.

\begin{proposition}[Defect formulas]\label{prop:defect}
Suppose that $n\geq3$. If $n$ is even and $N$ is contractible, then
$$
\G(\partial N)-|\Sph^{n-1}|
=
-
\frac{2|\Sph^{n-1}|}{|\Sph^n|}
\int_N\Pf(R)
-
\frac1{n-2}
\int_{\partial N}
L
-
\int_{\partial N} E_{n-1},
$$
where $E_{n-1}$ is as in Lemma~\ref{lem:TPf}; in particular $E_3=0$.
If instead $n$ is odd and $\partial N$ is diffeomorphic to
$\Sph^{n-1}$, then
$$
\G(\partial N)-|\Sph^{n-1}|
=
-\frac1{n-2}
\int_{\partial N}
L
-
\int_{\partial N} E_{n-1},
$$
where $E_{n-1}$ is as in Lemma~\ref{lem:linear-term}, applied to
the restriction of $R$ to $T\partial N$; in particular
$E_4=\Pf(R|_{T\partial N})$.
\end{proposition}

\begin{proof}
When $n$ is even, $\chi(N)=1$, since $N$ is contractible; so
integrating the formula of Lemma~\ref{lem:TPf} over $\partial N$ and
substituting into \eqref{eq:CGB} yields the first identity. When $n$
is odd, $\partial N$ has even dimension $n-1$ and
$\chi(\Sph^{n-1})=2$; so \eqref{eq:CGB} applied to $\partial N$
gives
$
\int_{\partial N}\Pf(R^{\partial N})=|\Sph^{n-1}|,
$
where $R^{\partial N}$ is the curvature tensor of $\partial N$. By
the Gauss equation,
$(R^{\partial N})_{ijk\ell}=R_{ijk\ell}+A_{ik}A_{j\ell}-A_{i\ell}A_{jk}$,
so \eqref{eq:RA} shows that $R^{\partial N}=R^A$. Expanding
$\Pf(R^A)$ by Lemma~\ref{lem:linear-term} in dimension $n-1$ then
yields the second identity.
\end{proof}

%%%%%%%%%%%%%%%%%%%%%%%%%%%%%%%%%%%%%%%
\section{Weighted Hsiung--Minkowski Inequalities}
\label{sec:newton}
%%%%%%%%%%%%%%%%%%%%%%%%%%%%%%%%%%%%%%%%%%

In this section, $\Gamma$ is a smooth convex hypersurface
in a Cartan--Hadamard manifold $(M^n,g)$, bounding a convex body $\Omega$. Let $\nu$ be the outward unit normal and $A$ be the shape
operator of $\Gamma$. For $0\leq j\leq n-1$, let
$$
\mathcal M_j=\mathcal M_j(\Gamma)\coloneqq\int_\Gamma\sigma_j(A)
$$
be the $j^{\mathrm{th}}$ \emph{total generalized mean curvature} of $\Gamma$.
Thus $\mathcal M_0=|\Gamma|$, $\mathcal M_1$ is the usual total mean curvature, and $\mathcal M_{n-1}=\G(\Gamma)$.
Here we establish  inequalities
relating the successive total mean curvatures $\mathcal M_j$. Their common ingredients are Stokes theorem, Hessian comparison, and a weighted Hsiung--Minkowski identity obtained from
Newton transformations. Applying this machinery to the squared distance
from a point in $\Omega$ yields inequalities involving the diameter,
while applying it to the distance from a point outside $\Omega$  yields
inequalities independent of the diameter.

%%%%%%%%%%%%%%%%%%%%%%%%%%%%%%%%%%%%%%%%%%
\subsection{Newton transformations}
%%%%%%%%%%%%%%%%%%%%%%%%%%%%%%%%%%%%%%%%%%

The integral identities used below are rooted in the works of Hsiung \cites{hsiung1954,hsiung1956} and Reilly \cite{reilly1973}; see 
\cites{kwong-lee-pyo2018,albuquerque2021,ghomi-spruck2023b} and references therein for background and other recent applications. These
identities generalize the classical
Minkowski formulas for convex hypersurfaces in Euclidean space \cites{montiel-ros2009,schneider2014}
and are naturally expressed in terms of Newton transformations.
For $0\leq j\leq n-1$, the $j^{\mathrm{th}}$ \emph{Newton transformation}
of $A$ is defined recursively by
$$
T_0\coloneqq I,
\qquad
T_j\coloneqq\sigma_j(A)I-AT_{j-1}.
$$
Since $A$ is a smooth self-adjoint $(1,1)$-tensor
field on $\Gamma$, so are $T_j$.

\begin{lemma}\label{lem:newton-identities}
For $1\leq j\leq n-1$,
\begin{equation}\label{eq:newton-traces}
\tr T_{j-1}=(n-j)\sigma_{j-1}(A),
\qquad
\tr(AT_{j-1})=j\sigma_j(A).
\end{equation}
If $A$ is positive semidefinite, then
$
0\leq T_{j-1}\leq\sigma_{j-1}(A)I.
$
\end{lemma}
\begin{proof}
In a principal frame $e_i$, with principal curvatures $\kappa_i$,
$$
T_{j-1}e_i
=
\sigma_{j-1}
(\kappa_1,\dots,\widehat{\kappa_i},\dots,\kappa_{n-1})e_i,
$$
which follows  by induction on
$j$, since
$$
\sigma_j(A)
-
\kappa_i\,\sigma_{j-1}(\kappa_1,\dots,\widehat{\kappa_i},\dots,\kappa_{n-1})
=
\sigma_j(\kappa_1,\dots,\widehat{\kappa_i},\dots,\kappa_{n-1}).
$$
Each monomial in $\sigma_{j-1}(A)$ occurs in $n-j$ of these
eigenvalues, while each monomial in $\sigma_j(A)$ occurs $j$ times
after multiplication by the corresponding principal curvature. This
proves \eqref{eq:newton-traces}. If $A$ is positive semidefinite, each
eigenvalue of $T_{j-1}$ lies between $0$ and $\sigma_{j-1}(A)$, which
proves the final assertion.
\end{proof}

The Hsiung--Minkowski formulas \cites{hsiung1954,montiel-ros1991} for a smooth closed hypersurface
$\Gamma\subset\R^n$ state that, for $1\leq j\leq n-1$,
$$
j\int_\Gamma\langle x,\nu\rangle\sigma_j(A)
=
(n-j)\int_\Gamma\sigma_{j-1}(A),
$$
where $\langle x,\nu\rangle$ is the support function of $\Gamma$. For $j=1$,  they reduce to the classical
Minkowski formula 
$
\int_\Gamma\langle x,\nu\rangle\sigma_1(A)
=
(n-1)|\Gamma|,
$
\cite[Thm.~6.11]{montiel-ros2009}; see \cite[Eq.~(5.60)]{schneider2014}.
The following proposition gives a weighted generalization of these formulas.
Indeed, in Euclidean space, choosing $f(x)\coloneqq|x|^2/2$ in the proposition recovers
the formulas above, since the gradient $\nabla f=x$, the Hessian $\nabla^2f$ restricts to
the identity on $T\Gamma$, and the Newton transformations are
divergence-free. For a $(1,1)$-tensor field $T$ on $\Gamma$, we write
$$
\diver T\coloneqq\sum_{i=1}^{n-1}\Big(\nabla^\Gamma_{e_i}T\Big)e_i,
$$
where $\nabla^\Gamma$ denotes the induced connection on $\Gamma$, and
$e_1,\dots,e_{n-1}$ is a local orthonormal frame on $\Gamma$; the
divergence of tangent vector fields on $\Gamma$ is defined similarly.
Furthermore, $\nabla_\Gamma f$ and $\nabla_\Gamma^2 f$ denote the
gradient and Hessian of the restriction of $f$ to $\Gamma$, computed
with respect to the induced metric and connection of $\Gamma$; in
particular, $\nabla_\Gamma f$ is the tangential component of
$\nabla f$.

\begin{proposition}[Weighted Hsiung--Minkowski identity]
\label{prop:newton-integration}
Let $f$ be a smooth function in a neighborhood of $\Gamma$. Then, for $1\leq j\leq n-1$,
\begin{equation}\label{eq:newton-integration}
j\int_\Gamma \langle\nabla f,\nu\rangle\sigma_j(A)
=
\int_\Gamma\tr\left(T_{j-1}\nabla^2f|_{T\Gamma}\right)
+
\int_\Gamma
\left\langle\diver T_{j-1},\nabla_\Gamma f\right\rangle.
\end{equation}
\end{proposition}
\begin{proof}
On $\Gamma$, we have
$
\nabla f=\nabla_\Gamma f+\langle\nabla f,\nu\rangle\nu.
$
So, for $X$, $Y$ tangent to $\Gamma$,
$$
\nabla^2f(X,Y)
=
\langle\nabla_X\nabla_\Gamma f,Y\rangle
+
\langle\nabla f,\nu\rangle\langle\nabla_X\nu,Y\rangle
=
\nabla_\Gamma^2f(X,Y)+\langle\nabla f,\nu\rangle A(X,Y),
$$
since $\nabla^2 f(X,Y)=\langle \nabla_X\nabla f, Y\rangle$, and $\nabla_\Gamma^2 f(X,Y)=\langle \nabla_X\nabla_\Gamma f, Y\rangle$.
Thus we obtain
$$
\nabla_\Gamma^2f=\nabla^2f|_{T\Gamma}-\langle\nabla f,\nu\rangle A.
$$
Since $T_{j-1}$ is self-adjoint,
$
\diver\big(T_{j-1}\nabla_\Gamma f\big)
=
\langle\diver T_{j-1},\nabla_\Gamma f\rangle
+
\tr\big(T_{j-1}\nabla_\Gamma^2f\big) 
$
by the product rule \cite[Eq.~(8.2)]{alias-lira-malacarne2006}. As $\Gamma$ is closed, the Stokes theorem gives
$$
0
=
\int_\Gamma\diver\big(T_{j-1}\nabla_\Gamma f\big)
=
\int_\Gamma\tr(T_{j-1}\nabla_\Gamma^2f)
+
\int_\Gamma
\langle\diver T_{j-1},\nabla_\Gamma f\rangle.
$$
Substitution of the preceding Hessian identity and
\eqref{eq:newton-traces} proves
\eqref{eq:newton-integration}.
\end{proof}

\subsection{Divergence estimates}
The last term in \eqref{eq:newton-integration} is absent from the
classical Hsiung--Minkowski formulas. To derive inequalities from \eqref{eq:newton-integration}, we therefore need to bound $\diver T_{j-1}$ in terms
of the ambient curvature and lower-order mean curvatures.
The next lemma provides this bound, together with the sharper estimate
for $T_1$ needed under curvature pinching. Here we regard $\Ric_R$ as a self-adjoint operator on $\R^n$, given by
$\langle\Ric_R(x),y\rangle\coloneqq\Ric_R(x,y)$. Let $(\cdot)^\top$ be the projection onto
$T\Gamma$.

\begin{lemma}\label{lem:newton-divergence}
We have $\diver T_0=0$, and, for $2\leq j\leq n-1$,
$$
|\diver T_{j-1}|
\leq
c_n|R|\sigma_{j-2}(A).
$$
Furthermore,
$
|\diver T_1|=|(\Ric_R(\nu))^\top|.
$
In particular, if
$
-1\leq K\leq-\lambda,
$
then
\begin{equation}\label{eq:div-T1-pinched}
|\diver T_1|
\leq
\frac{n-2}{2}(1-\lambda).
\end{equation}
\end{lemma}

\begin{proof}
Since $T_0=I$, its divergence vanishes. The Newton transformations can be written as
$$
(T_{j-1})_{ik}
=
\frac1{(j-1)!}
\sum
\delta^{k\,a_1\cdots a_{j-1}}_{i\,b_1\cdots b_{j-1}}
A_{a_1b_1}\cdots A_{a_{j-1}b_{j-1}},
$$
where $\delta^{i_1\cdots i_k}_{j_1\cdots j_k}$ is the generalized Kronecker delta,  and the sum ranges over indices from $1$ to $n-1$ 
\cites{reilly1973,alias-lira-malacarne2006}.
Differentiating in an orthonormal frame which is parallel at a given
point, and using the symmetry of the summands in the pairs
$(a_q,b_q)$, we obtain
\begin{equation}\label{eq:div-newton}
(\diver T_{j-1})_i
=
\frac1{(j-2)!}
\sum
\delta^{k\,a_1\cdots a_{j-1}}_{i\,b_1\cdots b_{j-1}}
(\nabla_{e_k}A)_{a_1b_1}
A_{a_2b_2}\cdots A_{a_{j-1}b_{j-1}}.
\end{equation}
Since the Kronecker delta is antisymmetric in $k,a_1$, only the part
of $(\nabla_{e_k}A)_{a_1b_1}$ antisymmetric in these indices
contributes. The Codazzi equation gives
\begin{equation}\label{eq:codazzi}
(\nabla_{e_k}A)_{a_1b_1}
-
(\nabla_{e_{a_1}}A)_{kb_1}
=
R(e_k,e_{a_1},e_{b_1},\nu).
\end{equation}
Thus this antisymmetric part is given by components of $R$. Choosing
the frame principal at the given point, each nonzero summand in
\eqref{eq:div-newton} is bounded in absolute value by
$|R|$ times a product of $j-2$ principal curvatures with distinct
indices, by the alternating property of the delta. Summing these
terms and using $\kappa_i\geq0$ gives
 the first stated estimate.

For $j=2$, \eqref{eq:div-newton} reduces to
$
\diver T_1=\nabla_\Gamma\sigma_1(A)-\diver A.
$
On the other hand, setting $b_1=k$ in \eqref{eq:codazzi} and summing
over $k$ yields the contracted Codazzi equation
$$
\diver A-\nabla_\Gamma\sigma_1(A)
=
\big(\Ric_R(\nu)\big)^\top.
$$
Hence
$
\diver T_1=-(\Ric_R(\nu))^\top.
$
Now suppose that $-1\leq K\leq-\lambda$, and write
$
R=R_{-1}+S.
$
Then the sectional curvature of $S$ lies between $0$ and
$1-\lambda$. Furthermore $\Ric_{R_{-1}}=-(n-1)\langle\cdot,\cdot\rangle$,
so the mixed components $\Ric_{R_{-1}}(\nu,X)$, for $X$ tangent to
$\Gamma$, vanish; hence $\Ric_R(\nu,X)=\Ric_S(\nu,X)$. For a unit tangent vector $X$, complete
$X,\nu$ to an orthonormal frame
$
X,e_2,\dots,e_{n-1},\nu.
$
Lemma~\ref{lem:berger} gives
$$
\left|\Ric_R(\nu,X)\right|
=
\left|
\sum_{i=2}^{n-1} S(e_i,\nu,e_i,X)
\right|
\leq
\frac{n-2}{2}(1-\lambda),
$$
which proves \eqref{eq:div-T1-pinched}.
\end{proof}

%%%%%%%%%%%%%%%%%%%%%%%%%%%%%%%%%%%%%%%%%%
\subsection{Inequalities involving diameter}\label{subsec:diameter}
%%%%%%%%%%%%%%%%%%%%%%%%%%%%%%%%%%%%%%%%%%

We first derive  inequalities involving the diameter of $\Gamma$ by
applying the weighted identity \eqref{eq:newton-integration} to $f\coloneqq r^2/2$, where $r$ denotes distance
from a point chosen in the interior of $\Omega$. For this choice, $f$
has uniformly positive Hessian, while $|\nabla f|=r$ is bounded by the
diameter. The resulting comparisons between successive total mean curvatures
will be used to control the terms containing at least two ambient
curvature factors in the proof of Theorem~\ref{thm:diameter}.

\begin{lemma}
\label{lem:interior-radial}
Suppose that $|R|\leq B^2$ on $\Omega$.
Then, for $1\leq j\leq n-1$,
\begin{equation}\label{eq:interior-radial}
(n-j)\mathcal M_{j-1}(\Gamma)
\leq
\Big(j\mathcal M_j(\Gamma)+c_nB^2\mathcal M_{j-2}(\Gamma)\Big)\diam(\Gamma),
\end{equation}
where the final term is omitted when $j=1$. Furthermore,
\begin{equation}\label{eq:interior-volume}
|\Omega|\leq\frac{\diam(\Gamma)}{n}|\Gamma|,
\end{equation}
independently of  $B$.
\end{lemma}

\begin{proof}
Choose a point $p$ in the interior of $\Omega$, let
$r\coloneqq\operatorname{dist}_M(p,\cdot)$, and set $f\coloneqq r^2/2$. Since $M$ is Cartan--Hadamard, Hessian comparison gives
$
\nabla^2f\geq g
$
 \cite[Lem.~6.2.5]{petersen2016}. Put $
d\coloneqq \diam(\Gamma).
$
Convexity gives
$
\diam(\Omega)=d,
$
and hence $|\nabla f|=r\leq d$ on $\Omega$.
Since $T_{j-1}\geq0$, Proposition \ref{prop:newton-integration} and Lemmas~\ref{lem:newton-identities}
 and \ref{lem:newton-divergence} give
$$
jd\mathcal M_j
\geq
j\int_\Gamma \langle\nabla f,\nu\rangle\sigma_j(A)
\geq
(n-j)\mathcal M_{j-1}
-
c_nB^2d\mathcal M_{j-2},
$$
where the final term is absent when $j=1$, since $\diver T_0=0$. This proves
\eqref{eq:interior-radial}.
Furthermore, $\Delta f\geq n$, so the Stokes theorem gives
$
n|\Omega|\leq\int_\Gamma \langle\nabla f,\nu\rangle.
$
Since $\langle\nabla f,\nu\rangle\leq|\nabla f|\leq d$, we obtain
\eqref{eq:interior-volume}.
\end{proof}

\begin{proposition}\label{prop:adjacent-mean-curvatures}
For $n\geq3$, under the hypotheses of Lemma~\ref{lem:interior-radial}, there exists $\varepsilon_n>0$ such that, if
$
B\diam(\Gamma)\leq\varepsilon_n,
$
then for $
2\leq s\leq\lfloor (n-1)/2\rfloor$,
$$
\mathcal M_{n-1-2s}(\Gamma)
\leq
c_n\diam(\Gamma)^{2s-2}\mathcal M_{n-3}(\Gamma),
$$
and
$$
|\Omega|\leq c_n\diam(\Gamma)^{n-2}\mathcal M_{n-3}(\Gamma).
$$
\end{proposition}

\begin{proof}
Put $d\coloneqq\diam(\Gamma)$. If $n=3$, the result follows directly
from \eqref{eq:interior-volume}. So suppose that $n\geq4$. We first
show that
$
\mathcal M_{j-1}\leq c_nd\mathcal M_j,
$
$
1\leq j\leq n-3.
$
For $j=1$, \eqref{eq:interior-radial} gives
$
|\Gamma|\leq c_nd\mathcal M_1.
$
Suppose inductively that
$
\mathcal M_{j-2}\leq c_nd\mathcal M_{j-1}.
$
Then \eqref{eq:interior-radial} yields
$$
\mathcal M_{j-1}
\leq
c_nd\mathcal M_j
+
c_n(Bd)^2\mathcal M_{j-1}.
$$
Choosing $\varepsilon_n$ sufficiently small and absorbing the final
term proves the claim. Iterating the claim from $\mathcal M_{n-1-2s}$ to $\mathcal M_{n-3}$
gives
$
\mathcal M_{n-1-2s}
\leq
c_nd^{2s-2}\mathcal M_{n-3}.
$
Similarly,
$
|\Gamma|\leq c_nd^{n-3}\mathcal M_{n-3}.
$
Combining this with \eqref{eq:interior-volume} gives
$
|\Omega|
\leq
c_nd|\Gamma|
\leq
c_nd^{n-2}\mathcal M_{n-3}.
$
\end{proof}
%%%%%%%%%%%%%%%%%%%%%%%%%%%%%%%%%%%%%%%%%%
\subsection{Inequalities independent of diameter}\label{subsec:nodiameter}
%%%%%%%%%%%%%%%%%%%%%%%%%%%%%%%%%%%%%%%%%%

We next derive  inequalities independent of the diameter of $\Gamma$.
Fix $p\notin\Omega$, set $r(\cdot)\coloneqq\operatorname{dist}_M(p,\cdot)$, and apply the weighted  identity \eqref{eq:newton-integration}
to $f\coloneqq r$. For this choice, Hessian comparison provides a lower bound for
$\nabla^2r$ independent of the size of $\Omega$. The resulting
inequalities will be used in dimensions $4$ and $5$.

\begin{lemma}
\label{lem:exterior-radial}
Suppose that $K\leq-a^2<0$. Then, for
$1\leq j\leq n-1$,
\begin{equation}\label{eq:exterior-radial}
j\int_\Gamma \langle\nabla r,\nu\rangle\sigma_j(A)
\geq
a\int_\Gamma
\big(n-1-j+ \langle\nabla r,\nu\rangle^2\big)\sigma_{j-1}(A)
-
\int_\Gamma|\diver T_{j-1}|.
\end{equation}
Moreover,
\begin{equation}\label{eq:radial-volume}
\int_\Gamma \langle\nabla r,\nu\rangle\geq (n-1)a|\Omega|.
\end{equation}
\end{lemma}

\begin{proof}
Since $M$ is Cartan--Hadamard and $p\notin\Omega$, the function $r$ is
smooth on $\Omega$, and Hessian comparison gives
\begin{equation}\label{eq:hessians}
\nabla^2r(X,X)
\geq
a\coth(ar)\big(|X|^2-\langle\nabla r,X\rangle^2\big)
\geq
a\big(|X|^2-\langle\nabla r,X\rangle^2\big),
\end{equation}
 see \cite[Thm.~6.4.3]{petersen2016}. Since
$|\nabla r|=1$ and
$\nabla r=\nabla_\Gamma r+\langle\nabla r,\nu\rangle\nu$ along
$\Gamma$, we have
$$
|\nabla_\Gamma r|^2=1-\langle\nabla r,\nu\rangle^2.
$$
Let $e_1,\dots,e_{n-1}$ be an orthonormal basis of eigenvectors
of $T_{j-1}$ on $T\Gamma$, with corresponding eigenvalues $\tau_i$,
which are nonnegative by Lemma~\ref{lem:newton-identities}. By \eqref{eq:hessians},
$$
\tau_i\,\nabla^2r(e_i,e_i)
\geq
a\,\tau_i\big(1-\langle\nabla r,e_i\rangle^2\big)
=
a\,\tau_i\big(1-\langle\nabla_\Gamma r,e_i\rangle^2\big),
$$
where the equality holds since $e_i$ are tangent to $\Gamma$. Summing
over $i$ then yields
$$
\tr\big(T_{j-1}\nabla^2r|_{T\Gamma}\big)
=
\sum_{i=1}^{n-1}\tau_i\,\nabla^2r(e_i,e_i)
\geq
a\Big(\tr T_{j-1}-T_{j-1}\big(\nabla_\Gamma r,\nabla_\Gamma r\big)\Big).
$$
Lemma~\ref{lem:newton-identities} gives
$\tr T_{j-1}=(n-j)\sigma_{j-1}(A)$, and
$T_{j-1}\leq\sigma_{j-1}(A)I$, which in turn yields
$T_{j-1}(\nabla_\Gamma r,\nabla_\Gamma r)\leq\sigma_{j-1}(A)|\nabla_\Gamma r|^2$.
So
$$
\tr\big(T_{j-1}\nabla^2r|_{T\Gamma}\big)
\geq
a\big(n-j-|\nabla_\Gamma r|^2\big)\sigma_{j-1}(A)
=
a\big(n-1-j+\langle\nabla r,\nu\rangle^2\big)\sigma_{j-1}(A).
$$
Substituting this estimate into \eqref{eq:newton-integration}, with
$f\coloneqq r$, and noting that
$
\big|\big\langle\diver T_{j-1},\nabla_\Gamma r\big\rangle\big|
\leq|\diver T_{j-1}|,
$
since $|\nabla_\Gamma r|\leq1$, yields \eqref{eq:exterior-radial}. Finally, taking the trace in \eqref{eq:hessians} gives
$\Delta r\geq(n-1)a$. So the Stokes
theorem yields
$$
\int_\Gamma\langle\nabla r,\nu\rangle
=
\int_\Omega\Delta r
\geq
(n-1)a|\Omega|,
$$
which is \eqref{eq:radial-volume}.
\end{proof}

Using the first two cases of \eqref{eq:exterior-radial}, for $j=1$ and
$j=2$, we now obtain:

\begin{proposition}\label{prop:low-order-radial}
Under the hypotheses of Lemma~\ref{lem:exterior-radial},
\begin{equation}\label{eq:M1-volume}
\mathcal M_1(\Gamma)\geq (n-1)^2a^2|\Omega|,
\end{equation}
for $n\geq3$. If $n\geq4$ and
$
-1\leq K\leq-\lambda<0,
$
then
\begin{equation}\label{eq:M2-pinched}
2\mathcal M_2(\Gamma)
\geq
\left(
(n-2)^2\lambda-\frac{n-2}{2}(1-\lambda)
\right)|\Gamma|.
\end{equation}
\end{proposition}

\begin{proof}
Set
$
U_0\coloneqq\int_\Gamma \langle\nabla r,\nu\rangle.
$
Since $\sigma_0(A)=1$, and $T_0=I$ is divergence free, the case $j=1$
of \eqref{eq:exterior-radial} reads
$$
\int_\Gamma\langle\nabla r,\nu\rangle\sigma_1(A)
\geq
a\left(
(n-2)|\Gamma|+\int_\Gamma\langle\nabla r,\nu\rangle^2
\right).
$$
Since $\langle\nabla r,\nu\rangle\leq1$ and $\sigma_1(A)\geq0$,
 the left-hand side is at most $\mathcal M_1$, while
$
\int_\Gamma\langle\nabla r,\nu\rangle^2\geq U_0^2/|\Gamma|
$
by the Cauchy--Schwarz inequality. Hence
$$
\mathcal M_1
\geq
a\left(
(n-2)|\Gamma|+\frac{U_0^2}{|\Gamma|}
\right)
\geq
(n-1)aU_0,
$$
where the last inequality holds since
$$
(n-2)|\Gamma|^2+U_0^2-(n-1)U_0|\Gamma|
=
\big((n-2)|\Gamma|-U_0\big)\big(|\Gamma|-U_0\big)
\geq0,
$$
because $U_0\leq|\Gamma|$.
Together with \eqref{eq:radial-volume},
this proves \eqref{eq:M1-volume}.

Next, to obtain \eqref{eq:M2-pinched}, set
$
U_1\coloneqq\int_\Gamma \langle\nabla r,\nu\rangle\sigma_1(A).
$
The first display above gives
$
U_1\geq a(n-2)|\Gamma|.
$
Furthermore, the case $j=2$ of \eqref{eq:exterior-radial}, together
with $\langle\nabla r,\nu\rangle\leq1$, $\sigma_2(A)\geq0$, and the
Cauchy--Schwarz inequality
$
\int_\Gamma\langle\nabla r,\nu\rangle^2\sigma_1(A)
\geq
U_1^2/\mathcal M_1,
$
yields
$$
2\mathcal M_2
\geq
a\left(
(n-3)\mathcal M_1+\frac{U_1^2}{\mathcal M_1}
\right)
-
\int_\Gamma|\diver T_1|
\geq
a(n-2)U_1
-
\int_\Gamma|\diver T_1|,
$$
where the last inequality holds, as above, since
$$
(n-3)\mathcal M_1^2+U_1^2-(n-2)U_1\mathcal M_1
=
\big((n-3)\mathcal M_1-U_1\big)\big(\mathcal M_1-U_1\big)
\geq0,
$$
because $n\geq4$ and $U_1\leq\mathcal M_1$. Consequently,
$$
2\mathcal M_2
\geq
a^2(n-2)^2|\Gamma|
-
\int_\Gamma|\diver T_1|.
$$
If
$
-1\leq K\leq-\lambda,
$
then we may take $a=\sqrt\lambda$, and
\eqref{eq:div-T1-pinched} yields \eqref{eq:M2-pinched}.
\end{proof}

\begin{note}
Since $\langle\nabla r,\nu\rangle\leq1$, \eqref{eq:radial-volume}
yields the linear isoperimetric inequality
$
|\Gamma|\geq(n-1)a|\Omega|,
$
which holds for domains in Cartan--Hadamard
manifolds with $K\leq-a^2$. This inequality was obtained by Yau
\cite{yau1975}, whose proof also rests on applying the divergence
theorem to $\nabla r$. See Burago--Zalgaller
\cite[Thm.~34.2.6]{burago-zalgaller1988} for another proof, via volume
comparison, and \cite[Rem.~34.2.8]{burago-zalgaller1988},
\cite{hoisington2021} for quantitative refinements. 
\end{note}

%%%%%%%%%%%%%%%%%%%%%%%%%%%%%%%%%%%%%%%%%%
\section{Convex Hypersurfaces of Small Diameter}\label{sec:diameter}
%%%%%%%%%%%%%%%%%%%%%%%%%%%%%%%%%%%%%%%%%%

Here we prove Theorem~\ref{thm:diameter}. The argument has two ingredients:
the Chern--Gauss--Bonnet expansion isolates a favorable term involving
$\mathcal M_{n-3}$, while the estimates of
Section~\ref{subsec:diameter} control all lower-order terms. The Chern--Gauss--Bonnet input, depending on the parity of $n$, is
supplied by Proposition~\ref{prop:defect}. First we establish a stronger quantitative result.

\begin{proposition}\label{prop:quantitative-diameter}
For every $n\geq4$ there exist constants $\varepsilon_n,c_n>0$ with the
following property. Let $M^n$ be a Cartan--Hadamard manifold,
$\Gamma\subset M$ a smooth convex hypersurface bounding the convex body $\Omega$, and
suppose that, on $\Omega$,
$
K\leq-a^2
$
and
$
|R|\leq B^2
$
for some $a,B>0$. If
$
B\diam(\Gamma)\leq\varepsilon_n,
$
then
$$
\G(\Gamma)-|\Sph^{n-1}|
\geq
\left(
\frac{a^2}{n-2}
-
c_nB^4\diam(\Gamma)^2
\right)
\mathcal M_{n-3}(\Gamma).
$$
\end{proposition}

Note that if $K\leq-a^2$ on $\Gamma$, then, since $\kappa_i\geq0$,
\begin{equation}\label{eq:linear-sign}
L
\leq
-a^2\sigma_{n-3}(A)
\end{equation}
at each point of $\Gamma$. We also need the following estimate:

\begin{lemma}\label{lem:GB-small}
Let $M$, $\Gamma$, $\Omega$, $a$,
and $B$ be as in Proposition~\ref{prop:quantitative-diameter}. Then
$$
\G(\Gamma)-|\Sph^{n-1}|
\geq
\frac{a^2}{n-2}\mathcal M_{n-3}(\Gamma)
-
c_n\sum_{s=2}^{\lfloor (n-1)/2\rfloor}
B^{2s}\mathcal M_{n-1-2s}(\Gamma)
-
c_nB^n|\Omega|,
$$
where the final term is omitted when $n$ is odd.
\end{lemma}

\begin{proof}
If $n$ is odd, then $\Gamma$, which bounds a convex body, is
diffeomorphic to $\Sph^{n-1}$, and
the odd case of Proposition~\ref{prop:defect} applies. If $n$ is
even, then $\Omega$ is contractible, and the even case of
Proposition~\ref{prop:defect} applies, where the homogeneity of the Pfaffian gives
$|\Pf(R)|\leq c_nB^n$. In either case, since $|R|\leq B^2$, the bounds on $E_{n-1}$ in
Lemmas~\ref{lem:linear-term} and \ref{lem:TPf} yield
$
|E_{n-1}|
\leq
c_n\sum_{s=2}^{\lfloor(n-1)/2\rfloor}B^{2s}\sigma_{n-1-2s}(A).
$
The result follows by integrating these estimates together with
\eqref{eq:linear-sign}.
\end{proof}

\begin{proof}[Proof of Proposition~\ref{prop:quantitative-diameter}]
Put $d\coloneqq \diam(\Gamma)$, and shrink $\varepsilon_n$ if
necessary so that $\varepsilon_n\leq1$. By
Proposition~\ref{prop:adjacent-mean-curvatures},
$$
B^{2s}\mathcal M_{n-1-2s}
\leq
c_nB^4d^2\mathcal M_{n-3},
\qquad
B^n|\Omega|
\leq
c_nB^4d^2\mathcal M_{n-3},
$$
where the first estimate holds for $s\geq2$, and the second is needed
only when $n$ is even. The result follows from
Lemma~\ref{lem:GB-small}.
\end{proof}

\begin{proof}[Proof of Theorem~\ref{thm:diameter}]
The pinching assumption implies
$
|R|\leq c_nb^2
$
by Lemma~\ref{lem:algebraic}, after scaling. Thus
Proposition~\ref{prop:quantitative-diameter} applies with
$B=\sqrt{c_n}\,b$, which, after renaming $\varepsilon_n$ and $c_n$, gives
$$
\G(\Gamma)-|\Sph^{n-1}|
\geq
\left(
\frac{a^2}{n-2}
-
c_nb^4\diam(\Gamma)^2
\right)
\mathcal M_{n-3}(\Gamma),
$$
provided that $b\diam(\Gamma)$ is sufficiently small.
 If
$
d\coloneqq\diam(\Gamma)\leq\delta_na/b^2,
$
then
$
bd\leq\delta_na/b\leq\delta_n
$
and
$
b^4d^2\leq\delta_n^2a^2.
$
Choosing $\delta_n$ sufficiently small therefore ensures both the
hypothesis of Proposition~\ref{prop:quantitative-diameter} and the
nonnegativity of the coefficient above. Since
$\mathcal M_{n-3}\geq0$, the result follows.
\end{proof}

\begin{note}
Choosing $\delta_n$ so that $c_n\delta_n^2\leq1/(2(n-2))$ in the above
proof yields the quantitative estimate
$$
\G(\Gamma)-|\Sph^{n-1}|
\geq
\frac{a^2}{2(n-2)}\,\mathcal M_{n-3}(\Gamma).
$$
Diameter-free defect bounds in dimensions $4$ and $5$ are given below in
Propositions~\ref{prop:dim4} and~\ref{prop:dim5}.
\end{note}

%%%%%%%%%%%%%%%%%%%%%%%%%%%%%%%%%%%%%%%%%%
\section{Diameter-Free Estimates in Low Dimensions}
\label{sec:low-dim}
%%%%%%%%%%%%%%%%%%%%%%%%%%%%%%%%%%%%%%%%%%

Here we prove Theorem~\ref{thm:pinched}. In both dimensions $4$ and $5$ the
Chern--Gauss--Bonnet formula reduces the desired inequality to a
comparison between lower-order geometric quantities, which is handled by estimates of Section~\ref{subsec:nodiameter}. More specifically, after replacing the metric $g$ of $M$ by $b^2g$, the curvature satisfies
$
-1\leq K\leq-\lambda_n.
$
Since $\G(\Gamma)$ is invariant under constant
rescaling, the result follows immediately from
Propositions~\ref{prop:dim4} and \ref{prop:dim5} below.

%%%%%%%%%%%%%%%%%%%%%%%%%%%%%%%%%%%%%%%%%%
\subsection{Dimension four}\label{sec:dim4}
%%%%%%%%%%%%%%%%%%%%%%%%%%%%%%%%%%%%%%%%%%

Here the only unfavorable term in the
Chern--Gauss--Bonnet formula is the interior Pfaffian integral. We
control it using the comparison between $\mathcal M_1$ and
$|\Omega|$ obtained in \eqref{eq:M1-volume}.

\begin{proposition}\label{prop:dim4}
Let $M^4$ be a Cartan--Hadamard manifold with
$
-1\leq K\leq-\lambda,
$
where $\lambda\geq\lambda_4$. Then every smooth convex hypersurface
$\Gamma\subset M$ satisfies
$$
\G(\Gamma)-|\Sph^3|
\geq
\frac{3\lambda^2-1}{6\lambda}\,
\mathcal M_1(\Gamma)
\geq0.
$$
\end{proposition}

\begin{proof}
Let $\Omega$ be the convex body bounded by $\Gamma$. Then $\Omega$ is
contractible. So the even case of Proposition~\ref{prop:defect}, where $E_3=0$ and
$2|\Sph^3|/|\Sph^4|=3/2$, together with \eqref{eq:linear-sign} for
$a=\sqrt\lambda$, gives
$$
\G(\Gamma)-|\Sph^3|
\geq
\frac{\lambda}{2}\mathcal M_1(\Gamma)
-
\frac32\int_\Omega\Pf(R).
$$
By Lemma~\ref{lem:pf-bound},
$
\int_\Omega\Pf(R)
\leq
|\Omega|.
$
On the other hand, \eqref{eq:M1-volume}, with $n=4$ and
$a=\sqrt\lambda$, gives
$
\mathcal M_1\geq9\lambda|\Omega|.
$
Consequently,
$$
\G(\Gamma)-|\Sph^3|
\geq
\left(
\frac{\lambda}{2}
-
\frac1{6\lambda}
\right)\mathcal M_1(\Gamma)
=
\frac{3\lambda^2-1}{6\lambda}\,
\mathcal M_1(\Gamma).
$$
The coefficient is nonnegative for
$\lambda_4\leq\lambda\leq1$.
\end{proof}

\begin{note}
The constant $\lambda_4=1/\sqrt3$ results from combining two estimates
which are sharp separately. Equality in Lemma~\ref{lem:pf-bound}
requires $R=R_{-1}$, whereas \eqref{eq:M1-volume} is asymptotically
sharp for large geodesic spheres in the space form of curvature
$-\lambda$. These sharpness mechanisms are different when
$\lambda<1$. In particular, the stronger estimate
$
\Pf(R)\leq1-(1+\Scal_R/12)^2
$
from Lemma~\ref{lem:pf-bound} has not been used, nor has any relation
between $\int_\Omega\Pf(R)$ and $\mathcal M_1(\Gamma)$. Exploiting
such information could potentially improve $\lambda_4$.
\end{note}

%%%%%%%%%%%%%%%%%%%%%%%%%%%%%%%%%%%%%%%%%%
\subsection{Dimension five}\label{sec:dim5}
%%%%%%%%%%%%%%%%%%%%%%%%%%%%%%%%%%%%%%%%%%

Here the intrinsic Chern--Gauss--Bonnet formula reduces
the problem to comparing $\mathcal M_2$ and $|\Gamma|$. The
required comparison is supplied by \eqref{eq:M2-pinched}.

\begin{proposition}\label{prop:dim5}
Let $M^5$ be a Cartan--Hadamard manifold with
$
-1\leq K\leq-\lambda,
$
where $\lambda\geq\lambda_5$. Then every smooth convex hypersurface
$\Gamma\subset M$ satisfies
$$
\G(\Gamma)-|\Sph^4|
\geq
\frac{7\lambda^2-\lambda-4}{4}\,
|\Gamma|
\geq0.
$$
\end{proposition}

\begin{proof}
Since $\Gamma$ bounds a convex body, it is diffeomorphic to
$\Sph^4$. So the case $n=5$ of Proposition~\ref{prop:defect}
gives
$$
\G(\Gamma)-|\Sph^4|
=
-\frac13\int_\Gamma L
-
\int_\Gamma\Pf(R|_{T\Gamma}).
$$
By \eqref{eq:linear-sign}, with $a=\sqrt\lambda$, and
Lemma~\ref{lem:pf-bound},
$
\G(\Gamma)-|\Sph^4|
\geq
\frac{\lambda}{3}\mathcal M_2
-
|\Gamma|.
$
Formula \eqref{eq:M2-pinched}, with $n=5$, gives
$
2\mathcal M_2\geq(21\lambda-3)|\Gamma|/2.
$
Therefore
$$
\G(\Gamma)-|\Sph^4|
\geq
\left(
\frac{\lambda(21\lambda-3)}{12}
-
1
\right)|\Gamma|
=
\frac{7\lambda^2-\lambda-4}{4}\,
|\Gamma|.
$$
The coefficient is nonnegative for
$\lambda_5\leq\lambda\leq1$.
\end{proof}

\begin{note}
Eliminating the divergence term in the proof of
\eqref{eq:M2-pinched} would replace the coefficient in
Proposition~\ref{prop:dim5} by $3\lambda^2/2-1$, lowering the required
pinching only to $\lambda\geq\sqrt{2/3}$. On the other hand,
\eqref{eq:M2-pinched} is not asymptotically sharp: for large geodesic
spheres in the space form of curvature $-\lambda$,
$
2\mathcal M_2(\Gamma)/|\Gamma|
\to
(n-1)(n-2)\lambda.
$
Thus $\lambda_5$ may potentially be improved by retaining information
discarded in the proof of Proposition~\ref{prop:low-order-radial},
such as the term
$\int_\Gamma\langle\nabla r,\nu\rangle^2$.
\end{note}

%%%%%%%%%%%%%%%%%%%%%%%%%%%%%%%%%%%%%
\section{Further Notes}
%%%%%%%%%%%%%%%%%%%%%%%%%%%%%%%%%%%%%

\begin{note}
The pinching hypotheses need only hold near $\Omega$ in the above results. In
Theorem~\ref{thm:diameter} the curvature bounds are required only on
$\Omega$, by Proposition~\ref{prop:quantitative-diameter}. In
Theorem~\ref{thm:pinched} the exterior center $p$ may be chosen
arbitrarily close to $\Gamma$; since the distance between geodesic
segments with a common endpoint is a convex function, the segments
joining $p$ to $\Omega$ then lie in any prescribed neighborhood of
$\Omega$, which is all that the proofs of the exterior radial
estimates require.
\end{note}

\begin{note}
The smoothness assumption on $\Gamma$ in
Theorems~\ref{thm:diameter} and \ref{thm:pinched} may be removed.
For an arbitrary convex hypersurface $\Gamma\subset M^n$, its outer parallel hypersurface $\Gamma^\varepsilon$, at distance $\varepsilon>0$, is 
$\mathcal{C}^{1,1}$ \cite{ghomi-spruck2022}. Thus $\G(\Gamma^\varepsilon)$ is well-defined by Rademacher's theorem, and we set
$$
\G(\Gamma)\coloneqq\lim_{\varepsilon\searrow0}\G(\Gamma^\varepsilon).
$$
This limit exists since $\varepsilon\mapsto \G(\Gamma^\varepsilon)$ is
nondecreasing. It also agrees with the usual definition when
$\Gamma$ is smooth, since $\G$ is continuous with respect to Hausdorff
distance \cite[Thm.~1.2]{ghomi2026-continuity}.
By the Greene--Wu smoothing procedure \cite{greene-wu1972}, as carried out in
\cite[Lem.~3.9]{ghomi2026-total}, there
exist smooth convex bodies $\Omega_i\subset\Omega$ which converge to
$\Omega$ in the Hausdorff metric. Set
$\Gamma_i\coloneqq\partial\Omega_i$. Then
$\diam(\Gamma_i)\leq\diam(\Gamma)$, the
ambient curvature hypotheses hold on $\Omega_i$, and
$\G(\Gamma_i)\to\G(\Gamma)$. Applying the theorems to $\Gamma_i$ and
passing to the limit yields the same conclusions for arbitrary convex
hypersurfaces.
\end{note}

\begin{note}
The diameter-free argument for pinched variable curvature, used in the
proof of Theorem~\ref{thm:pinched}, already encounters an obstruction
in dimension $6$. Suppose that $-1\leq K\leq-\lambda<0$. By Lemma~\ref{lem:algebraic}, multilinearity of the Pfaffian, and the
transgression formula \eqref{eq:transgression},
$$
|\Pf(R)+1|\leq c_6(1-\lambda),
\qquad
\left|E_5-\frac38\sigma_1(A)\right|
\leq c_6(1-\lambda)\sigma_1(A).
$$
Unlike in dimension $4$, nonpositive sectional curvature alone does
not determine the sign of the Pfaffian in dimension $6$. Indeed,
Geroch \cite{geroch1976} constructed a $6$-dimensional curvature
tensor with positive sectional curvatures and negative Pfaffian.
Since the Pfaffian is cubic in $R$, replacing $R$ by $-R$ gives
negative sectional curvatures and positive Pfaffian. Furthermore, $K_{ij}\leq-\lambda$ gives
$-L\geq\lambda\sigma_3(A)$. Hence Proposition~\ref{prop:defect}
yields
$$
\G(\Gamma)-|\Sph^5|
\geq
\left(\frac{15}{8}-c_6(1-\lambda)\right)|\Omega|
+\frac{\lambda}{4}\mathcal M_3
-\left(\frac38+c_6(1-\lambda)\right)\mathcal M_1.
$$
The case $j=1$ of \eqref{eq:exterior-radial} gives
$
\mathcal M_1\geq4\sqrt\lambda\,|\Gamma|.
$
Using \eqref{eq:div-T1-pinched}, the case $j=2$ gives
$$
\int_\Gamma\langle\nabla r,\nu\rangle\sigma_2(A)
\geq
\left(
\frac32\sqrt\lambda
-\frac{1-\lambda}{4\sqrt\lambda}
\right)\mathcal M_1.
$$
For the case $j=3$, note that
$2+\langle\nabla r,\nu\rangle^2\geq
3\langle\nabla r,\nu\rangle$. By
\eqref{eq:div-newton} and \eqref{eq:codazzi}, $\diver T_2$ involves
only curvature components with one normal and three tangential
entries. These components vanish for $R_{-1}$, and hence
$$
|\diver T_2|
\leq
c_6|R-R_{-1}|\sigma_1(A)
\leq
c_6(1-\lambda)\sigma_1(A)
$$
by Lemma~\ref{lem:algebraic}. Thus \eqref{eq:exterior-radial} gives
$$
\mathcal M_3
\geq
\sqrt\lambda
\int_\Gamma\langle\nabla r,\nu\rangle\sigma_2(A)
-c_6(1-\lambda)\mathcal M_1
\geq 
\big(3\lambda/2-c_6(1-\lambda)\big)\mathcal M_1.
$$
So we conclude that
$$
\G(\Gamma)-|\Sph^5|
\geq
\left(\frac{15}{8}-c_6(1-\lambda)\right)|\Omega|
-c_6(1-\lambda)\mathcal M_1.
$$
For $\lambda$ close to $1$, the coefficient of
$|\Omega|$ is positive; however, the negative term
$c_6(1-\lambda)\mathcal M_1$ cannot in general be absorbed into the
first term, since $\mathcal M_1(\Gamma)/|\Omega|$ is unbounded.
Indeed, for geodesic balls $B_r\subset\mathbf H^n$,
$
\mathcal M_1(\partial B_r)/|B_r|
\sim c_n/r^2\to\infty
$,
as
$
r\to0.
$
The same phenomenon occurs for thin tubes about a geodesic segment.

\end{note}

\bibliography{references}

\end{document}